\documentclass[11pt, twoside]{article}
\usepackage{a4, 
latexsym, amsmath, amsxtra, amssymb, exscale, multicol, dsfont, delarray}
\usepackage[amsmath,amsthm,thref]{ntheorem}
\usepackage[latin1]{inputenc}
\usepackage{epsfig}
\usepackage[dvipsnames]{pstricks}
\usepackage{pst-node, pst-tree}
\usepackage{verbatim}
\usepackage{fancyhdr}

\usepackage[round]{natbib}

\newcommand{\R}{\mathds{R}}

\newcommand{\N}{\mathds{N}}

\newcommand{\cE}{\mathcal{E}}

\newcommand{\cL}{\mathcal{L}}
\newcommand{\cM}{\mathcal{M}}

\newcommand{\cV}{\mathcal{V}}

\newcommand{\1}{\mathds{1}}
\newcommand{\epsi}{\ensuremath{\varepsilon}}

\newcommand{\vect}[3]{\ensuremath{#1_{#2},\ldots,#1_{#3}}}

\renewcommand{\rho}{\varrho}

\newcommand{\de}{\mathrm{d}}

\newcommand{\eqd}{\stackrel{d}{=}}
\newcommand{\eqn}{\begin{equation}}
\newcommand{\nqe}{\end{equation}}
\newcommand{\pmx}{\begin{pmatrix}}
\newcommand{\xmp}{\end{pmatrix}}
\newcommand{\bmx}{\begin{bmatrix}}
\newcommand{\xmb}{\end{bmatrix}}

\DeclareMathOperator{\Var}{Var}

\theoremstyle{definition}
\newtheorem{definition}{Definition}[section]
\theoremstyle{plain}
\newtheorem{satz}[definition]{Theorem}
\newtheorem{lemma}[definition]{Lemma}
\newtheorem{korollar}[definition]{Corollary}
\theoremstyle{remark}
\newtheorem{bem}[definition]{Remark}

\newcommand{\qed}{~\hfill$\Box$}
\newcommand{\BOX}{\ensuremath\Box}
\renewenvironment{proof}[1]{{\vskip\baselineskip\noindent\textbf{Proof#1.}}}%
{\origqed\vskip\baselineskip\gdef\origqed{\hspace*{.1pt}\hspace*{\fill}\BOX}}

\makeatletter
\def\@tagformdelimstart{(}%
\def\@tagformdelimend{)}%
\def\@tagformdel{%
  \gdef\@tagformdelimstart{}%
  \gdef\@tagformdelimend{}%
}
\def\@tagformset{%
  \gdef\@tagformdelimstart{(}%
  \gdef\@tagformdelimend{)}%
}
\def\tagform@#1{%
   \maketag@@@{\@tagformdelimstart\ignorespaces#1\unskip%
   \@@italiccorr\@tagformdelimend}\@tagformset}
\def\origqed{\hspace*{.1pt}\hspace*{\fill}\BOX}
\def\qed{\ifmmode%
 \@tagformdel%
 \tag{\BOX}%
 \else%
 \hspace*{.1pt}\hspace*{\fill}\BOX%
 \fi%
 \gdef\origqed{}}
\makeatother

\allowdisplaybreaks
\title{\bfseries \boldmath %
On the asymptotic internal path length and the asymptotic Wiener index of random split trees}
\author{G\"otz Olaf Munsonius\\[1.2ex]
J.W. Goethe University\\
Institute of Mathematics\\
60054 Frankfurt a.M., Germany\\
munsonius@math.uni-frankfurt.de\\[2ex]
}

\begin{document}

\maketitle

\begin{abstract}
The random split tree introduced by \citet{devroye_99} is considered.
We derive a second order expansion for the mean of its internal path length and furthermore obtain a limit law by the contraction method.
As an assumption we need the splitter having a Lebesgue density and mass in every neighborhood of $1$.
We use properly stopped homogeneous Markov chains, for which limit results in total variation distance as well as renewal theory are used.
Furthermore, we extend this method to obtain the corresponding results for the Wiener index.
\end{abstract}

\noindent%

\normalsize
\section{Introduction}\label{sec-1}
The random split tree introduced by \citet{devroye_99} is a general tree model which for special choices of its parameters covers various random trees that are fundamental in Computer Science for their use as data structures, e.g.\ binary search trees, quadtrees, $m$-ary search trees, simplex trees, tries etc.
Many characteristic quantities of these trees such as node depths, height, path length or other distance measures between nodes describe the complexity of algorithms that make use of the trees.
In the probabilistic analysis of algorithms the asymptotic behavior of such quantities is studied for this reason.
Whereas often such characteristic quantities are studied one by one for each tree Devroye's idea was to derive universal results valid for the whole class of his split tree model.

We recall the definition of the split tree from \citet{devroye_99}.
Four parameters $b,s,s_0,s_1\in\N_0$ are given where $b\ge 2$ is the branching factor, $s>0$ is the vertex capacity and $s_0$ and $s_1$ satisfy the two conditions
\[
0\le s_0\le s,\qquad 0\le bs_1\le s+1-s_0.
\]
Furthermore, a random vector $\cV=(V_1,\ldots,V_b)\in[0,1]^b$ with $\sum_{k=1}^bV_k=1$ is given.
The random split tree of size $n$ is obtained by distributing $n$ balls to the nodes of the infinite $b$-ary tree according to the following procedure.
For a node $u$ of the $b$-ary tree let $C(u)$ denote the number of balls already assigned to this node and $N(u)$ be the number of balls associated to any node in the subtree rooted at this node.
For each node $u$ take an independent copy $\cV^{(u)}=(V^{(u)}_1,\ldots,V^{(u)}_b)$ of the random vector $\cV$.
Initially, there are no balls (i.e.\ $C(u)=0$ for all $u$) distributed.
The balls are added to the tree sequentially.
Adding a ball to a tree rooted at $u$ proceeds as follows:
\begin{enumerate}
\item If $u$ is not a leaf (i.e.\ $C(u)<N(u)$), choose child $i$ with probability $V^{(u)}_i$, increment $N(u)$ by $1$ and recursively add the ball to the subtree rooted at child $i$.
\item If $u$ is a leaf and $C(u)=N(u)<s$, then add the ball to $u$ and stop.
$C(u)$ and $N(u)$ are incremented by $1$.
\item If $u$ is a leaf but $C(u)=N(u)=s$, we set $N(u)=s+1$ and $C(u)=s_0$, place $s_0\le s$ randomly selected balls at $u$, give $s_1$ randomly selected balls to each of the $b$ children of $u$ and set $C(v)=s_1=N(v)$ for all children $v$ of $u$.
After that, we add each of the remaining $s+1-s_0-bs_1\ge 0$ balls one by one randomly and independently to the subtree rooted at child $i$ with probability $V^{(u)}_i$ by applying the procedure recursively.
\end{enumerate}

Usually, one assumes that $V_i\eqd V_1=:V$ for all $i=2,\ldots,b$ where $V$ is called the splitter and its distribution is called the splitting distribution.
By $\eqd$ it is denoted that left and right hand side have identical distributions.
Whenever the functional under consideration is independent of the tree ordering, this assumption does not mean any loss of generality.
This can be seen by a random permutation argument, already stated in \citet{devroye_99}.
In this paper we need some additional assumption:\\[1.5ex]
\textbf{General assumption:} Throughout this paper we assume that the distribution of $V$ has a Lebesgue density $f_V$ and that for the distribution function we have $F_V(x)<1$ for all $x<1$.\\[1.5ex]
As mentioned in the beginning, the random split tree models many common random trees.
For instance, choosing $s=s_0=b-1$ for some $b\ge 2$, $s_1=0$ and $V=\min\{U_1,\ldots,U_{b-1}\}$ where $U_1,\ldots,U_{b-1}$ are independent random variables uniformly distributed on $[0,1]$ one gets the random $b$-ary search tree.
The random median-of-$(2k+1)$ binary search tree can be realized by setting $b=2$, $s=2k$, $s_0=1$, $s_1=k$ and $V=\mathrm{median}(U_1,\ldots,U_{2k+1})$.
Also some digital data structures are covered by the split tree model.
For $V$ uniformly distributed on the deterministic set $\{p_1,\ldots,p_b\}$, $s=1$ and $s_1=0$ one obtains in the case $s_0=0$ the trie and in the case $s_0=1$ the digital search tree.
In Table 1 in \citet{devroye_99} more examples of important tree models are listed with the corresponding choices of the parameters.

The general assumption and with it the results of this paper hold true for many of these examples as random binary search trees, random $b$-ary search trees, random quadtrees, random median-of-$(2k+1)$ binary search trees, random simplex trees, (extended) AB trees and random $m$-grid trees.
Whereas the results are not applicable to the common digital data structures as tries and digital search trees.\\[1ex]

The depth of the $n$-th ball in a random split tree, denoted by $D_n$, is the number of edges on the path from the ball to the root of the tree.
The internal path length of balls in the split tree is the sum of all depths of balls and is denoted by $P_n$ for the tree with $n$ balls.
Thus, we have
\[
P_n=\sum_{k=1}^nD_k.
\]

The asymptotic expansion of the expectation of $P_n$ was investigated for $m$-ary search trees in \citet{mahmoud_86}, for random quadtrees by \citet{flajolet_labelle_laforest_salvy_95} and for the median of $(2k+1)$-binary search tree by \citet{chern_hwang_01} and \citet{roesler_01}.
In \citet{holmgren_10} the internal path length of random split trees is considered under the assumption that the splitting distribution is non-lattice.
The first term and an upper bound of the second term of the asymptotic mean are derived using renewal theory.

Limit theorems for the distribution of the path length are proved for the random binary search tree in \citet{regnier_89} and \citet{roesler_91} and for the random recursive tree in \citet{dobrow_fill_99}.

Using the contraction method, \citet[Theorem 5.1]{neininger_rueschendorf_99} showed a universal limit theorem for the internal path length of random split trees under the assumption that the asymptotic expansion of the expectation of the internal path length is of the form
\eqn\label{as_form}
E[P_n]=d_1n\log n+d_2n+o(n)
\nqe
as $n\to\infty$.
Therefore, it is of interest to characterize all splitting distributions providing an asymptotic expectation of the form \eqref{as_form}.
The first result of this paper is the following.
\begin{satz}\label{th}
Let $P_n$ denote the internal path length in a random split tree of size $n$ with branching factor $b$ where the one-dimensional marginal distribution $V$ of the splitting vector fulfills the general assumption.
Then there exists a constant $c_p\in\R$ with
\[
 E[P_n]=\frac 1{\mu}n\log n+c_pn+o(n)
\]
as $n\to\infty$ where $\mu=-bE[V\log V]$.
\end{satz}

To state the result which follows from the combination of the limit theorem from \citet{neininger_rueschendorf_99} with Theorem \ref{th} we introduce some notation.
By $\mathcal M_{0,2}$ we denote the set of centered probability measures on $\R$ with finite second moments.
We denote the distribution of a random variable $X$ by $\mathcal L(X)$ or $P^X$.
The Wasserstein-metric $\ell_2$ on $\mathcal M_{0,2}$ is defined by
\begin{equation}
\ell_2(\nu_1,\nu_2):=\inf\{\|X-Y\|_2:\mathcal L(X)=\nu_1, \mathcal L(Y)=\nu_2\}
\end{equation}
where the $L_2$-norm $\|\cdot\|_2$ is given by \begin{math}\|X\|_2=(E[\|X\|^2])^{1/2}\end{math}.
For random variables $X$ and $Y$ we set $\ell_2(X,Y):=\ell_2(\cL(X),\cL(Y))$.
It is well known that convergence with respect to the metric $\ell_2$ (denoted by $\stackrel{\ell_2}{\longrightarrow}$) is equivalent to weak convergence plus convergence of the second moments (see e.g.\ \citet{bickel_freedman_81}).

\begin{korollar}\label{cor-ipl}
Let $P_n$ denote the internal path length in a random split tree of size $n$ where the one-dimensional marginal distribution of the splitting vector $(V_1,\ldots,V_b)$ fulfills the general assumption.
Define $X_n:=(P_n-E[P_n])/n$.
Then the following holds true:
\begin{enumerate}
\item As $n\to\infty$ we have $\ell_2(X_n,X)\to 0$ where $\cL(X)$ is the in $\mathcal M_{0,2}$ unique solution of the fixed point equation
\[
X\eqd\sum_{k=1}^bV_kX^{(k)}+1+\frac 1{\mu}\sum_{k=1}^bV_k\log V_k
\]
where $\mu:=-bE[V_1\log V_1]$, $\mathcal L(X^{(k)})=\mathcal L(X)$ for all $k=1,\ldots,b$ and $X,X^{(1)},\ldots,X^{(b)},(V_1,\ldots,V_b)$ are independent.
\item
In particular, the convergence in a) implies 
\[
\Var(P_n)=\sigma^2n^2+o(n^2)
\]
with
\[
\sigma^2=\left(\frac 1{\mu^2}E\left[\left(\sum_{k=1}^bV_k\log V_k\right)^{\kern-0.6ex 2\kern+0.6ex}\right]-1\right)\left(1-\sum_{k=1}^bE\left[V_k^2\right]\right)^{\kern-0.6ex-1}.
\]
\item Exponential moments exist and converge,
\[
E[\exp(\lambda X_n)]\to E[\exp(\lambda X)],\qquad\lambda\in\R.
\]
\item For all $k\in\N$ we have as $n\to\infty$,
\[
P(|P_n-E[P_n]|\ge\epsi E[P_n])=O(n^{-k}).
\]
\end{enumerate}
\end{korollar}

\begin{bem}
The tail bound given in d) is known not to be sharp in particular examples.
\citet{mcdiarmid_hayward_96} and \citet{fill_janson_02} give a more precise bound for the random binary search tree.
\end{bem}

The Wiener index of a random split tree is defined as the sum of the distances between all unordered pairs of balls, where the distance between two balls is given by the minimum number of edges connecting the nodes which are associated to the balls.
For trees, the two dimensional vector consisting of the Wiener index and the internal path length suffices a recursion formula similar to that of the latter one.
Using this recursion formula, \citet{neininger_02} proved a limit theorem for the Wiener index of the random binary search tree and the random recursive tree by the use of the multivariate contraction theorem.
In a final remark, \citet{neininger_02} mentioned that a limit theorem for the Wiener index of the general split tree can be proved in a similar way after determining the asymptotic expansion of its expectation sufficiently well.

We prove this asymptotic expansion and use the contraction method to obtain the limit theorem for the Wiener index of random split trees which fulfil the general assumption.

\begin{satz}\label{th-wiener}
Let $W_n$ denote the Wiener index in a random split tree of size $n$ with branching factor $b$ where the one-dimensional marginal distribution $V$ of the splitting vector fulfills the general assumption.
Then there exists a constant $c_w\in\R$ with
\[
 E[W_n]=\frac 1{\mu}n^2\log n+c_wn^2+o(n)
\]
as $n\to\infty$ where $\mu=-bE[V\log V]$.
\end{satz}

We denote by $\cM_{0,2}^2$ the set of centered probability measures on $\R^2$ with finite second moments.
The Wasserstein-metric $\ell_2$ on the set $\cM_{0,2}^2$ is defined similarly to the one-dimensional case.
\begin{satz}%
\label{th-limit-wiener}
Let $(W_n,P_n)$ denote the  vector consisting of the Wiener index and the internal path length of a random split tree of size $n$ with branching factor $b$ where the one-dimensional marginal distribution of the splitting vector $(V_1,\ldots,V_b)$ fulfills the general assumption.
Then the following holds true:
\begin{enumerate}
\item We have as $n\to\infty$,
\[
\ell_2\left(\left(\frac{W_n-E[W_n]}{n^2},\frac{P_n-E[P_n]}{n}\right),(W,P)\right)\to 0
\]
where $(W,P)$ is the unique distributional fixed-point of the map $T:\mathcal M_{0,2}^2\to\mathcal M_{0,2}^2$ given for $\nu\in\mathcal M_{0,2}^2$ by
\[
T(\nu):=\mathcal L\left(\sum_{i=1}^b\begin{bmatrix} V_i^2&V_i(1-V_i)\\0&V_i\end{bmatrix}\begin{pmatrix} X_1^{(i)}\\X_2^{(i)}\end{pmatrix}+\begin{pmatrix} b_1^\ast\\b_2^\ast\end{pmatrix}\right)
\]
with
\[
\begin{pmatrix} b_1^\ast\\b_2^\ast\end{pmatrix}=\frac 1\mu \sum_{i=1}^b V_i\log V_i\begin{pmatrix} 1\\1\end{pmatrix}+\begin{pmatrix}(1+c_p-c_w)\left(1-\sum_{i=1}^bV_i^2\right)\\1\end{pmatrix}
\]
where $\mathcal L(X^{(i)})=\nu$ for $X^{(i)}:=(X_1^{(i)},X_2^{(i)})$, and $X^{(1)},\ldots,X^{(b)},D,Z$ are independent.
\item
In particular, the convergence in a) implies 
\[
\Var(W_n)=\sigma^2n^4+o(n^4)
\]
with some constant $\sigma^2>0$.
\end{enumerate}
\end{satz}

\begin{bem}
The constant $\mu=-bE[V\log V]$ in the first order terms of the expectations of the internal path length and of the Wiener index appears already in the results about the height and depth in \citet{devroye_99}.
There, the explicit values of this constant for the individual splitting distributions are given in Table 2.
\end{bem}

\begin{bem}
Besides the internal path length for the balls considered here, there is also the internal path length for the nodes where the depths of all nodes are summed up.
Since there can be up to $s$ balls in one node, these two path lengths may differ.
In \citet{holmgren_10}, the relation between the two versions is investigated.
Let $N_n$ denote the number of nodes in the random split tree with $n$ balls.
Assuming that the distribution of $-\log V$ is non-lattice, $P(V=1)=P(V=0)=0$ and
\eqn\label{holmgren}
E[N_n]=\alpha n+O\left(\frac n{(\log n)^{1+\epsi}}\right)
\nqe
for some constant $\alpha>0$ and $\epsi>0$,
\citet{holmgren_10} showed that Theorem \ref{th} implies the similar asymptotic behavior for the internal path length for the nodes in that random split tree.
This finally yields the general limit theorem for the internal path length for the nodes in split trees which additionally fulfil equation \eqref{holmgren}.
For instance, \citet{mahmoud_pittel_89} showed the stronger result $E[N_n]=\alpha n+O(n^{1-\epsi})$ in the case of the $b$-ary search tree.

It seems that there are no results on the corresponding alternative version of the Wiener index in terms of the node-to-node distances.
\end{bem}

The internal path length and the Wiener index have been considered also for random trees that do not belong to the class of split trees.
A universal limit law for the path length of simply generated trees is proved in \citet{janson_03} where the limit distribution is given as a function of the Brownian excursion. Furthermore, the moments of the limit are derived.
For the class of random increasing trees, which covers in particular the random recursive tree and the plane oriented recursive tree, the second order asymptotic of the expectation of the internal path length is derived in \citet{bergeron_flajolet_salvy_07}.
In \citet{munsonius_rueschendorf_10} the asymptotic behavior of the expectation and a limit theorem for the internal path length of random $b$-ary trees with weighted edges is proved.
By special choices of the edge weights, the analogous results are obtained for the class of random linear recursive trees, which encompasses in particular the random plane oriented recursive tree.
Tail bounds for the Wiener index of random binary search trees have been considered by \citet{alikhan_neininger_07}.

For a random split tree with $n$ balls we denote by $I_n=(I_{n,1},\ldots,I_{n,b})$ the vector of the sizes of the subtrees, i.e. the number of balls assigned to nodes in the subtrees, rooted at the children of the root.
By the construction of the split tree it follows that $I_n$ is conditionally given $\cV^{(\mathrm{root})}=(v_1,\ldots,v_b)$ multinomial distributed $M(n-s_0-bs_1;v_1,\ldots,v_b)$.
Thus, under the assumption that $V_i\eqd V_1=:V$ for all $i=2,\ldots,b$ we obtain
\eqn\label{eq-distr-I}
P(I_{n,i}=k+s_1)=\int_0^1\binom{\eta_n}{k}x^k(1-x)^{\eta_n-k}\de P^V(x),
\nqe
where we set $\eta_n:=n-s_0-bs_1$.
Throughout this paper, $\mathrm{Bin}(m,x)$ denotes a random variable with binomial distribution with parameters $m\in\N$ and $x\in[0,1]$.

The proofs of Theorem \ref{th} and Theorem \ref{th-wiener} are based on a method developed in \citet{bruhn_96} for recurrences where the toll function is bounded.
In Section \ref{sec-bruhn}, we recall definitions and results of \citet{bruhn_96} and extend his method to the case of an unbounded toll function.
We check the conditions of this method in the case of the random split tree in Section \ref{sec-splittree}.
Section \ref{sec-ipl} is devoted to the application in the case of the internal path length and the proof of Theorem \ref{th}.
In Section \ref{sec-wi} we give the proofs of Theorem \ref{th-wiener} and Theorem \ref{th-limit-wiener} concerning the Wiener index.

\paragraph{Acknowledgement.}
The author is grateful to Ralph Neininger for several hints to literature and for comments to previous versions of this paper and to Nicolas Broutin for helpful discussions and making a preliminary manuscript of the paper \citet{broutin_holmgren_11}
on the internal path length of split trees available to him.
Furthermore, he thanks an unknown referee for valuable suggestions for improvement of the paper.

\section{The setting of Bruhn}\label{sec-bruhn}
Starting from recursion formulas of the form
\[
H_n=\sum_{k=0}^{n-1}\nu_n(\{k\})H_k+r(n)
\]
where $\nu_n$ is a probability measure on $\{0,\ldots,n-1\}$ for all $n\in\N$, the main idea of \citet{bruhn_96} is to define a homogeneous Markov chain $(S_t)_{t\in\N}$ with state space $\cE=\{-\log n:n\in\N\}\cup\{1\}$ where the transition probabilities are given for $n>0$ by
\[
P(S_1=x\mid S_0=-\log n)=
\begin{cases}
\nu_n(\{e^{-x}\}),&\text{ for $x\in\{-\log(n-1),\ldots,-\log 1\}$}\\
\nu_n(\{0\}),&\text{ for $x=1$}
\end{cases}
\]
and $P(S_1=1\mid S_0=1)=1$.
Now, let $\sigma(n_1):=\inf\{t\mid S_t>-\log n_1\}$ be the stopping time when the Markov chain exceeds $-\log n_1$ for $n_1\in\N$.
Then, Bruhn proved the representation formula given in the following Lemma. (Since the PhD-thesis of Bruhn seems to be not available in English, the proofs of \citet{bruhn_96} are stated in Appendix \ref{app-b}.)

We denote by $Y_t:=S_t-S_{t-1}$ the increments of $S$.
For $x\in\cE$ we write $P_x(\cdot)$ in short for $P(\cdot\mid S_0=x)$ and correspondingly $E_x[\cdot]$ for the expectation with respect to the measure $P_x$.
We denote by $F_x$ the distribution function of $P_x^{S_1-x}$, i.e.\ $F_x(y)=P(S_1-x\le y\mid S_0=x)$.

\begin{lemma}\label{le-representation}
Let $H_n$ be a sequence of real numbers satisfying 
\[
H_n=\sum_{k=0}^{n-1}\nu_n(\{k\})H_k+r(n)
\]
for some function $r$.
Then it is for any $n_1\in\N$ with the notations above
\eqn\label{zentral}
H_n=E_{-\log n}H_{\exp(-S_{\sigma(n_1)})}+E_{-\log n}\sum_{t=0}^{\sigma(n_1)-1}r(\exp(-S_t)).
\nqe
\end{lemma}

To analyze the Markov chain $(S_t)_{t\in\N}$ we consider in the following a general state space $\cE\subset\R$.

\begin{definition}\label{def-mea}
The Markov chain $(S_t)_{t\in\N_0}$ is said to be an \emph{AR-process} (approximate renewal) if the state space $\cE$ has no lower bound, the increments $Y_t:=S_t-S_{t-1}$ are strictly positive, $F_x$ converges in distribution as $x\to-\infty$ to a distribution function $F$, i.e.\ for all points $t$ where $F$ is continuous it is
\[
\lim_{x\to-\infty}F_x(t)=F(t),
\]
and $0<\int t\,\de F(t)<\infty$.
\end{definition}

For $a\in\R_-$ we define $\bar F_a:\R\to[0,1]$ by $\bar F_a(t):=\inf_{x\le a}F_x(t)$ and $\underline F_a:\R\to[0,1]$ by $\underline F_a(t):=\sup_{x\le a}F_x(t)$.
\begin{definition}
The set of distributions $\{F_x\}$ fulfills the \emph{integrability condition} if
\[
\lim_{a\to-\infty}\int x\,\de\bar F_a(x)=\int x\,\de F(x).
\]
\end{definition}
In the case of an AR-process, the theorem of dominated convergence implies that the integrability condition is equivalent to
\eqn\label{ic}
\int x\,\de\bar F_a(x)<\infty
\nqe
for some $a\in\R$.

The first summand in \eqref{zentral} can be handled by considering the distribution of $S_{\sigma(n_1)}$.
The following key result is implicitly given in \citet{roesler_01} in a more general setting.
The essential part of \citet{roesler_01} which gives the proof is stated in Appendix \ref{app-a} in a self-contained way.
For probability measures $P$ and $Q$, let $d_\mathrm{TV}(P,Q)$ denote their total variation distance.
Moreover, we define $\tau(d):=\inf\{t:S_t\ge d\}$.
\begin{lemma}\label{le-roesler}
Let $(S_t)_{t\in\N}$ be an AR-process which fulfills the integrability condition with a discrete state space $\cE$.
If there exist $\epsi>0$, $x_0\in\R_-$ and $K>0$ such that for all $x,y\le x_0$ with $|x-y|\le K$ we have
\begin{align}\label{cond-tv}
d_\mathrm{TV}\left(P_x^{S_1},P_y^{S_1}\right){}<{}&2(1-\epsi)&&\text{and}&\lim_{x_0\to-\infty}\inf_{z<y\le x_0}P_z(S_{\tau(y)}-y\le K){}>{}&0,
\end{align}
then it holds for any $a\in\R_-$
\[
 \lim_{x_0\to-\infty}\sup_{x,y\le x_0}d_\mathrm{TV}\left(P_x^{S_{\tau(a)}},P_y^{S_{\tau(a)}}\right)=0.
\]
\end{lemma}

The asymptotic behavior of the second summand in \eqref{zentral} can be analyzed by using the elementary renewal theorem.
Since the Markov chain $(S_t)_{t\in\N}$ is not a  renewal process, we couple it with three renewal processes using the functions $F$, $\bar F_a$ and $\underline F_a$.
Because of the convergence $\lim_{x\to-\infty}F_x(t)=F(t)$, the functions $\bar F_a$ and $\underline F_a$ are again distribution functions.

Considering the AR-process $(S_t)$ from above, there exists a sequence of independent random variables $(U_r)_{r\in\N}$ uniformly distributed on $[0,1]$ such that
\[
Y_t=F_{S_{t-1}}^{-1}\circ U_t
\]
for all $t\in\N$.

For $a\in\R$ we define three renewal processes $\bar S^{(a)}$, $\underline S^{(a)}$ and $\tilde S$ by $\bar S_0^{(a)}=\underline S_0^{(a)}=\tilde S_0=S_0$ and the i.i.d.\ increments $\bar Y_r^{(a)}$, $\underline Y_r^{(a)}$ and $\tilde Y_r$ given by
\begin{align*}
\bar Y_t^{(a)}&{}:={}\bar F_a^{-1}\circ U_t,&\underline Y_t^{(a)}&{}:={}\underline F_a^{-1}\circ U_t&\text{and}&&\tilde Y_r&{}:={}F^{-1}\circ U_t.
\end{align*}
Thus, for all $t\in\N$ we have $\underline Y_t^{(a)}\le S_{t}-S_{t-1}\le \bar Y_t^{(a)}$ whenever $S_{t-1}\le a$.

Moreover, for each $t\in\N$ the sequence $\bar Y^{(a)}_t$ is decreasing and $\underline Y_t^{(a)}$ is increasing as $a\to-\infty$.
Both sequences converge almost surely to $\tilde Y_r$.

Finally, we define the following stopping times for $a,d\in\R$:
\begin{align*}
\tau(d){}:={}&\inf\{t:S_t\ge d\},&\gamma(d){}:={}&\inf\{t:S_t-S_0\ge d\},\\
\bar\tau^{(a)}(d){}:={}&\inf\{t:\bar S_t^{(a)}\ge d\},&\bar\gamma^{(a)}(d){}:={}&\inf\{t:\bar S_t^{(a)}-\bar S_0^{(a)}\ge d\},\\
\underline\tau^{(a)}(d){}:={}&\inf\{t:\underline S_t^{(a)}\ge d\},&\underline\gamma^{(a)}(d){}:={}&\inf\{t:\underline S_t^{(a)}-\underline S_0^{(a)}\ge d\},\\
&\text{and}&
\tilde\gamma(d){}:={}&\inf\{t:\tilde S_t-\tilde S_0\ge d\}.
\end{align*}

Using the renewal process $(\bar S_t)_{t\in\N}$, \citet{bruhn_96} shows the following result. 
(The proof is given in Appendix\ref{app-b}.)
\begin{lemma}[\citet{bruhn_96}, Lemma 3.4]~\label{le-bruhn-1}
Consider an AR-process $(S_t)$ with the notations above.
Then there exist a real number $a_\ast$ and a positive real number $\hat u(a_\ast)$ such that for all measurable functions $l:\R\to\R_+$, all real numbers $y,z$ and all $x\in\cE$ with $x<y<z<a_\ast$ we have
\[
E_x\left[\sum_{t=\tau(y)}^{\tau(z)-1}l(S_t)\right]\le\hat u(a_\ast)\sum_{n=\lfloor y\rfloor}^{\lceil z\rceil}\sup_{t\in(n-1,n]}l(t).
\]
\end{lemma}

To investigate also recurrences where the toll function $r$ is not bounded as it is for example in the case of the Wiener index, we complete the results of Bruhn by the following lemma and corollary.
\begin{lemma}\label{le-hauptteil}
It holds for all decreasing continuous functions $l:\R\to\R_+$ and any $d\in\R_+$
\begin{align*}
\lim_{a\to-\infty}E\left[\sum_{t=1}^{\bar\gamma^{(a)}(d)}l\left(\bar S^{(a)}_t-\bar S^{(a)}_0\right)\right]{}={}&\lim_{a\to-\infty}E\left[\sum_{t=1}^{\underline\gamma^{(a)}(d)}l\left(\underline S^{(a)}_t-\underline S^{(a)}_0\right)\right]\\
{}={}&E\left[\sum_{t=1}^{\tilde\gamma(d)}l\left(\tilde S_t-\tilde S_0\right)\right]<\infty.
\end{align*}
\end{lemma}

\begin{proof}{}
First, we consider the sequence $(\bar S_t^{(a)})$.
By the construction we know that for each $s,t\in\N$ the mapping $a\mapsto\bar Y^{(a)}_s$ and thus the mapping $a\mapsto\bar S^{(a)}_t-\bar S^{(a)}_0$ are decreasing and converge almost surely to $\tilde Y_s$ and $\tilde S_t-\tilde S_0$ as $a\to-\infty$.
This yields that for $d\in\R$ the mapping $a\mapsto\bar\gamma^{(a)}(d)$ is increasing and bounded from above by $\tilde\gamma(d)$.
It is easy to see that $\bar\gamma^{(a)}(d)\to\tilde\gamma(d)$ almost surely as $a\to-\infty$.
Since $\bar\gamma^{(a)}(d)\in\N$ for all $a\in\R$ and $l$ is continuous, we obtain as $a\to-\infty$ almost surely
\[
\sum_{t=1}^{\bar\gamma^{(a)}(d)}l\left(\bar S^{(a)}_t-\bar S^{(a)}_0\right)\to\sum_{t=1}^{\tilde\gamma(d)}l\left(\tilde S^{(a)}_t-\tilde S^{(a)}_0\right).
\]
Furthermore, the left hand side is increasing as $a\to-\infty$ and
\[
E\left[\sum_{t=1}^{\tilde\gamma(d)}l\left(\tilde S^{(a)}_t-\tilde S^{(a)}_0\right)\right]\le l(0)E[\tilde\gamma(d)]
\]
where we use that $l$ is decreasing.
The positivity of $\tilde Y_s$ ensures by \citet[][Chapter II, Theorem 3.1]{gut_88} that $E[\tilde\gamma(d)]<\infty$ and the claim follows for the first sum.

With the same arguments, we have
\eqn\label{help}
\sum_{t=1}^{\underline\gamma^{(a)}(d)}l\left(\underline S^{(a)}_t-\underline S^{(a)}_0\right)\to\sum_{t=1}^{\tilde\gamma(d)}l\left(\tilde S^{(a)}_t-\tilde S^{(a)}_0\right)
\nqe
almost surely as $a\to-\infty$ and the left hand side is decreasing.
It is
\[
E\left[\sum_{t=1}^{\underline\gamma^{(a)}(d)}l\left(\underline S^{(a)}_t-\underline S^{(a)}_0\right)\right]\le l(0)E[\underline\gamma^{(a)}(d)].
\]
The monotone convergence theorem provides $\lim_{a\to-\infty}E[\underline Y_t^{(a)}]=E[\tilde Y_t]>0$.
Thus, $E[\underline Y_t^{(a)}]>0$ for $a\in\R$ small enough and the elementary renewal theorem \citep[see e.g.][Section II.4]{gut_88} implies $E[\underline\gamma^{(a)}(d)]<\infty$.
So, the claim follows from \eqref{help} by the monotone convergence theorem.
\end{proof}

Choosing $l(x)=\exp(-\alpha x)$ with $\alpha>0$ yields the following result.
\begin{korollar}\label{co-hauptteil}
For $\alpha,d>0$ there exists a constant $c\in\R$ such that for each $\epsi>0$ there exists $n_0\in\N$ with
\[
\frac 1{n^{\alpha}}E_{-\log n}\left[\sum_{t=0}^{\tau(-\log n+d)}\exp(-\alpha S_t)\right]\in(c-\epsi,c+\epsi)
\]
for all $n\ge n_0$.
\end{korollar}

\begin{proof}{}
By construction we have for $-\log n+d\le a$
\begin{align*}
\sum_{t=0}^{\bar\gamma^{(a)}(d)}\exp(-\alpha(\bar S^{(a)}_t-\bar S^{(a)}_0)){}\le{}& \sum_{t=0}^{\gamma(d)}\exp(-\alpha(S_t-S_0)\\
\le{}& \sum_{t=0}^{\underline\gamma(d)}\exp(-\alpha(\underline S^{(a)}_t-\underline S^{(a)}_0)).
\end{align*}
For $\epsi>0$, Lemma \ref{le-hauptteil} provides $a_\ast\in\R$ such that for all $a<a_\ast$ we have
\[
\left|E\left[\sum_{t=0}^{\bar\gamma(d)}\exp\left(-\alpha \left(\bar S^{(a)}_t-\bar S^{(a)}_0\right)\right)\right]-E\left[\sum_{t=0}^{\underline\gamma(d)}\exp\left(-\alpha\left(\underline S^{(a)}_t-\underline S^{(a)}_0\right)\right)\right]\right|<\epsi.
\]
We choose $n_0$ such that $-\log n_0+d\le a_\ast$.
Since we have for $n\ge n_0$
\[
E_{-\log n}\left[\sum_{t=0}^{\tau(-\log n+d)}\exp(-\alpha S_t)\right]=n^\alpha E_{-\log n}\left[\sum_{t=0}^{\gamma(d)}\exp\left(-\alpha(S_t-S_0)\right)\right]
\]
the claim follows using Lemma \ref{le-hauptteil} once more.
\end{proof}

\section{Recurrences for the random split tree}\label{sec-splittree}
We consider a random split tree with the notation as introduced in Section \ref{sec-1} and set $\nu_n(\{k\}):=b\frac knP(I_{n,1}=k)+\frac {s_0}{n}\1_{\{k=n-s_0\}}$.
This function $\nu_n$ defines a probability measure on the set $\{0,\ldots,n-s_0\}$.
This is seen by summing up all values
\begin{align*}
 \sum_{k=0}^{n-s_0}\nu_n(\{k\}){}={}&b\frac 1nE[I_{n,1}]+\frac {s_0}{n}\\
={}&\frac{n-s_0}{n}+\frac{s_0}{n}\\
={}&1.
\end{align*}

For the rest of the paper, we consider the Markov chain $(S_t)_{t\in\N}$ from Section \ref{sec-bruhn} where the transition probabilities are given by this special choice of $\nu$.
In this section, we prove that for this choice the conditions of the Lemmata of the previous section are fulfilled.

\subsection{The distribution of the subtreesize}
When doing this, we frequently use the fact that the size of the first subtree rescaled properly converges.
\begin{lemma}\label{le-conv-I}
For $\epsi>0$ we have
\[
P\left(\left|\frac{I_{n,1}}{n}-V\right|\ge\epsi\right)\le 2\exp\left(-\frac{n\epsi^2}{4}\left(1+O\left(\frac 1n\right)\right)\right).
\]
In particular, this yields
\[
E\left[\left|\frac{I_{n,1}}{n}-V\right|\right]=O\left(n^{-\frac 13}\right).
\]
\end{lemma}

\begin{proof}{}
Starting from the distribution of $I_{n,1}$ given in \eqref{eq-distr-I} we obtain by Bernstein's inequality
\begin{align*}
P\left(\left|\frac{I_{n,1}}{n}-V\right|\ge\epsi\right){}={}&\int_0^1 P\left(\left|\mathrm{Bin}(\eta_n,x)-nx\right|\ge n\epsi\right)\de P^V(x)\\
\le{}&2\exp\left(-\frac{n\epsi^2}{4}\left(1+O\left(\frac 1n\right)\right)\right).
\end{align*}
Since it is $|I_{n,1}/n-V|\le 1$, this yields for the expectation
\begin{align*}
E\left[\left|\frac{I_{n,1}}{n}-V\right|\right]{}={}&E\left[\left(\1_{\left\{\left|\frac{I_{n,1}}{n}-V\right|\le n^{-\frac 13}\right\}}+\1_{\left\{\left|\frac{I_{n,1}}{n}-V\right|> n^{-\frac 13}\right\}}\right)\left|\frac{I_{n,1}}{n}-V\right|\right]\\
\le{}&n^{-\frac 13}+2\exp\left(-\frac{n^{1/3}}{4}\left(1+O\left(\frac 1n\right)\right)\right)\\
={}&O\left(n^{-\frac 13}\right).
\end{align*}
\end{proof}

At this point, we prove some asymptotic expansions needed later.
\begin{lemma}\label{le-exp-I}
For the size of the first subtree $I_{n,1}$ in a random split tree with splitting distribution $V$ it holds
\[
E[I_{n,1}^2]=E[V^2]n^2+o(n^2),
\]
\[
E[I_{n,1}\log I_{n,1}]=\frac 1b n\log n+E[V\log V]n+o(n)
\]
and
\[
E[I_{n,1}^2\log I_{n,1}]=E[V^2]n^2\log n+E[V^2\log V]n^2+o(n^2).
\]
\end{lemma}

\begin{proof}{}
It is
\begin{align}\label{exp-I-1}
E[I_{n,1}^2]{}={}&\int_0^1E[\mathrm{Bin}(\eta_n,x)^2]\de P^V(x)\notag\\
{}={}&\int_0^1(\eta_nx(1-x)+\eta_n^2x^2)\de P^V(x)\notag\\
={}&E[V^2]n^2+o(n^2).
\end{align}
Furthermore, we have by Lemma \ref{le-conv-I} $I_{n,1}/n\to V$ in probability.
Since $x\mapsto x^k\log x$ is bounded on the interval $[0,1]$, we obtain for $k=1,2$
\[
E\left[\frac{I_{n,1}^k}{n^k}\log \frac{I_{n,1}}{n}\right]\to E[V^k\log V].
\]
This implies
\[
E\left[I_{n,1}^k\log \frac{I_{n,1}}{n}\right]=E[V^k\log V]n^k+o(n^k).
\]
On the other hand we have
\[
E\left[I_{n,1}^k\log \frac{I_{n,1}}{n}\right]=E\left[I_{n,1}^k\log I_{n,1}\right]-E\left[I_{n,1}^k\right]\log n.
\]
The claims follow with result \eqref{exp-I-1} since we have $E[I_{n,1}]=(n-s_0)/b$.
\end{proof}

\subsection{The Markov chain for the random split tree}
Now, we consider the Markov chain from Section \ref{sec-bruhn} with the transition probabilities $\nu_n(\{k\})=b\frac knP(I_{n,1}=k)+\frac {s_0}{n}\1_{\{k=n-s_0\}}$.
\begin{lemma}\label{le-5.2}
The process $(S_t)_{t\in\N_0}$ is an AR-process and the corresponding set of distributions $\{F_x\}$ fulfills the integrability condition.
\end{lemma}
\begin{proof}{}
Since $\nu_n$ is a probability measure on the set $\{0,\ldots,n-s_0\}$ we have $Y_t>0$ for all $t$.
For $x=-\log n$ we have by dominated convergence and Lemma \ref{le-conv-I} for any $y\in\R$
\begin{align*}
F_x(y){}={}&P(Y_1\le y\mid S_0=x)\\
={}&\sum_{k\in \N:-\log\frac kn\le y}\nu_n(\{k\})\\
={}&\sum_{k\in \N:-\log\frac kn\le y}b\frac kn P(I_{n,1}=k)+\frac{s_0}{n}\1_{\{n-s_0\ge e^{-y}n\}}\\
={}&bE\left[\frac {I_{n,1}}{n}\1_{\{-\log(I_{n,1}/n)\le y\}}\right]+\frac{s_0}{n}\1_{\{n-s_0\ge e^{-y}n\}}\\
\xrightarrow{n\to\infty}{}&bE[V\1_{\{-\log V\le y\}}]=:F(y).
\end{align*}
Moreover, we obtain with Fubini's Theorem
\begin{align*}
\int_0^\infty t\,\de F(t){}={}&\int_0^\infty(1-F(t))\,\de t\\
={}&\int_0^\infty bE\left[V\1_{\{-\log V>t\}}\right]\,\de t\\
={}&-bE[V\log V].
\end{align*}
This yields $0<\int t\,\de F(t)<\infty$.

It remains to show the integrability condition, which means
\[
\int t\,\de\bar F_a(t)<\infty
\]
for an $a\in\R$ and $\bar F_a(t):=\inf_{x\le a}F_x(t)$.
Using again Fubini's Theorem we obtain
\begin{align*}
\int t\,\de \bar F_a(t){}={}&\int\int_0^\infty\1_{[0,t]}(y)\,\de y\de\bar F_a(t)\\
={}&\int_0^\infty\int \1_{[y,\infty)}(t)\,\de\bar F_a(t)\de y.
\end{align*}
Since 
\[
\int \1_{[y,\infty)}(t)\,\de\bar F_a(t)=\lim_{z\to\infty}\bar F_a(z)-\bar F_a(y)\le 1-\bar F_a(y)
\]
it follows for $a=-\log m$
\begin{align*}
\int t\,\de \bar F_a(t){}\le{}&\int_0^\infty\sup_{x\le a}(1-F_x(y))\,\de y\\
\le{}&\int_0^\infty b\sup_{n\ge m}\underbrace{E\left[\frac{I_{n,1}}{n}\1_{\{-\log(I_{n,1}/n)>y\}}\right]}_{\le e^{-y}}\,\de y\\
\le{}&\int_0^\infty be^{-y}\,\de y\\
<{}&\infty.
\end{align*}
\end{proof}

\begin{lemma}\label{le-tv}
The process $(S_t)_{t\in\N}$ fulfills the assumptions of Lemma \ref{le-roesler}.
\end{lemma}
\begin{proof}{}
In the previous proof we have already shown that $(S_t)_{t\in\N}$ is an AR-process, which fulfills the integrability condition.
The state space $\cE=\{-\log n\mid n\in\N\}\cup\{1\}$ is discrete.
It remains to show conditions \eqref{cond-tv}.
Let $x=-\log n$ and $y=-\log m$ with $m<n$.
It is 
\eqn\label{tv-dist}
d_\mathrm{TV}\left(P_x^{S_1},P_y^{S_1}\right)=2-2\sum_{z\in E}\min\{P_x(S_1=z),P_y(S_1=z)\}.
\nqe
We will show that there exists $0<\tilde\alpha<\tilde\beta<1$ such that for $n$ large enough
\eqn\label{anfang_tv}
0<\sum_{k=\lceil \tilde\alpha n\rceil+s_1}^{\lfloor\tilde\beta n\rfloor+s_1}\min\left\{\int_0^1\binom{\eta_l-1}{k-s_1-1}z^{k-s_1}(1-z)^{\eta_l-k+s_1}\de P^V(z)\mid l=n,m\right\}.
\nqe
For $k=cn+o(n)$ with $c\in(0,1)$ and $n\to\infty$ we have
\begin{align*}
\lefteqn{P_{x}(S_1=-\log k)}\qquad\\
={}&b\frac knP(I_{n,1}=k)+\frac{s_0}{n}\1_{\{k=n-s_0\}}\\
={}&b\frac k{k-s_1}\frac{\eta_n}{n}\int_0^1\frac{k-s_1}{\eta_n}P(\mathrm{Bin}(\eta_n,z)=k-s_1)\de P^V(z)+\frac{s_0}{n}\1_{\{k=n-s_0\}}\\
={}&(1+o(1))b\int_0^1\binom{\eta_n-1}{k-s_1-1}z^{k-s_1}(1-z)^{\eta_n-k+s_1}\de P^V(z)+o(1).
\end{align*}
Hence, inequality \eqref{anfang_tv} and equation \eqref{tv-dist} will imply
\[
d_\mathrm{TV}\left(P_x^{S_1},P_y^{S_1}\right)<2-2\epsi
\]
for some $\epsi>0$.
The condition $|x-y|\le K$ is equivalent to $m\ge e^{-K}n$.

By the general assumption, the distribution of $V$ has a Lebesgue density $f_V$.
Thus, there exists $\tilde z\in(0,1)$ with $f_V(\tilde z)>0$.
Theorem 3 in Section 1.7.2 of \citet{evans_gariepy_92} (which is a Corollary from the Lebesgue-Besicovitch Differentiation Theorem) implies that we can find a non-empty interval $(\alpha,\beta)\subset(0,1)$ and $\epsi_1>0$  such that $\lambda(\{z\in(\alpha,\beta)\mid f_V(z)<\epsi_1\})=0$ with $\lambda$ the Lebesgue measure.
Now, we can choose some $\epsi_2>0$ and $K>0$ with $\tilde \alpha:=\alpha+\epsi_2<e^{-K}(\beta-\epsi_2)=:\tilde \beta$.

We will show that for $n$ large enough, for all $k\in[\tilde\alpha n+s_1,\tilde\beta n+s_1]\cap\N$ and for all $l\in[e^{-K}n,n]\cap\N$  it holds
\[
\int_0^1\binom{\eta_l-1}{k-s_1-1}z^{k-s_1}(1-z)^{\eta_l-k+s_1}\de P^V(z)\ge \frac 12\epsi_1\frac 1{n+1}.
\]

First, we consider the function $g:z\mapsto z^{k-s_1}(1-z)^{\eta_l-k+s_1}$.
Integration by parts yields
\eqn\label{hilf_1}
\int_0^1z^{k-s_1}(1-z)^{\eta_l-k+s_1}\de z=\frac {k-s_1}{(\eta_l+1)\eta_l}\binom{\eta_l-1}{k-s_1-1}^{-1}.
\nqe
For $k=c\eta_l+s_1$ the function $g$ reaches its maximum at $\hat z=c$, is increasing on the interval $[0,c]$ and decreasing on $[c,1]$.
Therefore, we have for any $\epsi_3\in(0,c\wedge (1-c))$
\[
\int_0^{c-\epsi_3}z^{c\eta_l}(1-z)^{(1-c)\eta_l}\de z\le \tilde g_c(\epsi_3)^{\eta_l}
\]
and
\[
\int_{c+\epsi_3}^1z^{c\eta_l}(1-z)^{(1-c)\eta_l}\de z\le \tilde g_c(-\epsi_3)^{\eta_l}
\]
where we set $\tilde g_c(\epsi_3):=(c-\epsi_3)^c(1-c+\epsi_3)^{(1-c)}$.
Stirling's formula yields
\[
\binom{\eta_l-1}{c\eta_l-1}^{-1}\sim\sqrt{2\pi c(1-c)}\frac 1c((1-c)^{1-c}c^c)^{\eta_l}\sqrt{\eta_l}=\sqrt{2\pi \frac {1-c}{c}} \tilde g_c(0)^{\eta_l}\sqrt{\eta_l}.
\]
Considering the derivative of $\tilde g_c$ in a neighborhood of $0$, we obtain $\tilde g_c(x)<\tilde g(0)\le1$ for all $x\not=0$ with $|x|$ small enough.
More precisely, for all $c\in[\tilde \alpha,\tilde\beta]$ and $\epsi_3>0$ small enough we have $\tilde g_c(\epsi_3)/\tilde g_c(0)\in(0,C)$ for some constant $C<1$.
Thus, for $\epsi_3>0$ small enough and $l$ large enough we have
\[
\int_0^{c-\epsi_3}z^{c\eta_l}(1-z)^{(1-c)\eta_l}\de z\le\frac 14\binom{\eta_l-1}{c\eta_l-1}^{-1}\frac c{\eta_l+1}
\]
and
\[
\int_{c+\epsi_3}^1z^{c\eta_l}(1-z)^{(1-c)\eta_l}\de z\le\frac 14\binom{\eta_l-1}{c\eta_l-1}^{-1}\frac c{\eta_l+1}.
\]
Together with \eqref{hilf_1}, this implies for some $0<\epsi_3<\epsi_2$, $l$ large enough and $c\in[\tilde\alpha,\tilde\beta]$ with $c\eta_l\in\N$
\[
\int_{c-\epsi_3}^{c+\epsi_3}\binom{\eta_l-1}{c\eta_l-1}z^{c\eta_l}(1-z)^{(1-c)\eta_l}\de z\ge \frac 12\frac c{\eta_l+1}.
\]
We obtain for any $k\in[\tilde\alpha n+s_1,\tilde\beta n+s_1]\cap\N$ and $l\in[e^{-K}n,n]\cap\N$ when $n$ is large enough
\begin{align*}
\lefteqn{\int_0^1\binom{\eta_l-1}{k-s_1-1}z^{k-s_1}(1-z)^{\eta_l-k+s_1}\de P^V(z)}\qquad\\
{}\ge{}&\epsi_1\int_{\alpha}^{\beta}\binom{\eta_l-1}{k-s_1-1}z^{k-s_1}(1-z)^{\eta_l-k+s_1}\de z\\
\ge{}&\frac 12\epsi_1\frac {\tilde\alpha}{\eta_l+1}\\
\ge{}&\frac 12\epsi_1\frac {\tilde\alpha}{n+1}.
\end{align*}
This finally yields \eqref{anfang_tv}:
\begin{align*}
\lefteqn{\sum_{k=\lceil \tilde\alpha n\rceil+s_1}^{\lfloor\tilde\beta n\rfloor+s_1}\kern-1ex \min\left\{\int_0^1\binom{\eta_l-1}{k-s_1-1}z^{k-s_1}(1-z)^{\eta_l-k+s_1}\de P^V(z)\mid l=n,m\right\}}\qquad\\
\ge{}&\frac 12\epsi_1\left(\tilde\beta-\tilde\alpha\right)\tilde\alpha+o(1)\hspace*{7.3cm}\\
>{}&0.
\end{align*}

As in the proof of Lemma \ref{le-5.2} we see that
\begin{align*}
P_x(S_{\tau(y)}-y\le K){}\ge{}&\inf_{x<y}P_x(S_1-S_0\le K)\\
={}&\bar F_y(K)\\
\xrightarrow{y\to-\infty}{}&bE\left[V\1_{\{V\ge e^{-K}\}}\right].
\end{align*}
Since $e^{-K}<1$ the general assumption $F_V(x)<1$ for all $x<1$ implies $bE\left[V\1_{\{V\ge e^{-K}\}}\right]>0$.
This shows the second condition and the proof is finished.
\end{proof}

\section{The internal path length}\label{sec-ipl}
After these preliminaries, we are now able to prove Theorem \ref{th}.
To show Theorem \ref{th} we have to prove that the sequence
\[
H_n:=\frac{E[P_n]-\mu^{-1}n\log n}{n}
\]
converges.
The internal path length $P_n$ suffices a recursive representation \citep[see e.g.][equation (50)]{neininger_rueschendorf_99} from where we get
\[
E[P_n]=\sum_{k=0}^{n-s_0}bP(I_{n,1}=k)E[P_k]+n-s_0.
\]
This recursion formula implies
\[
H_n=\sum_{k=0}^{n-s_0}\nu_n(\{k\})H_k+t(n)-\frac{s_0}{n}H_{n-s_0}
\]
with $t(n)=\frac 1n(n-s_0-\mu^{-1}n\log n+b\mu^{-1}E[I_{n,1}\log I_{n,1}])$ and $\nu_n(\{k\})$ as in the previous section.

From the result about the mean of the depth in \citet{devroye_99} we know $H_n\le C\log n$ for some constant $C>0$.
Therefore, we have for any $\delta_1\in(0,1)$
\[
\frac{s_0}{n}H_{n-s_0}\le Cs_0\frac {\log n}{n}=O\left(\frac 1{n^{\delta_1}}\right).
\]
Furthermore, because of $n=bE[I_{n,1}]+s_0$, we have
\[
t(n)=1-\frac 1{E[V\log V]}E\left[\frac{I_{n,1}}{n}\log\frac{I_{n,1}}{n}\right]+O\left(\frac 1{\sqrt{n}}\right).
\]
The function $x\mapsto x\log x$ is H\"older continuous. Using this and considering the rate of convergence of $E[|\frac {I_{n,1}}{n}-V|]$ in Lemma \ref{le-conv-I} we obtain with Jensen's inequality $t(n)=O(n^{-\,\delta_2})$ for some $\,\delta_2>0$.
Taking all this into account, we get
\eqn\label{rec-hn}
H_n=\sum_{k=0}^{n-s_0}\nu_n(\{k\})H_k+r(n) 
\nqe
where $r(n)=O(n^{-\,\delta})$ for some $\,\delta\in(0,1]$.

\begin{proof}{ of Theorem \ref{th}}
Equation \eqref{rec-hn} shows that the condition of Lemma \ref{le-representation} is fulfilled.
Thus, we start with the representation of
\[
H_n=\frac{E[P_n]-\mu^{-1}n\log n}{n}
\]
from there and show that $(H_n)_{n\in\N}$ is a Cauchy sequence.
Let $\epsi>0$ be given.

For the second term in \eqref{zentral} we keep in mind that we have already shown $|r(n)|\le Cn^{-\delta}$ for some constant $0<C<\infty$ and $\delta\in(0,1]$.
We define $l:\R\to\R^+$ by $l(x):=\exp(\delta x)$.
As in the proof of Theorem 4.2 in \citet{bruhn_96} we obtain with Lemma \ref{le-bruhn-1} for $n_1\in\N$ with $-\log n_1\le a_\ast$
\begin{align*}
\bigg|E_{-\log n}\sum_{t=0}^{\sigma(n_1)-1}r(\exp(-S_t))\bigg|{}\le{}&E_{-\log n}\sum_{t=0}^{\sigma(n_1)-1}Cl(S_t)\\
\le{}&C\hat u(a_\ast)\sum_{n=-\infty}^{\lceil-\log n_1\rceil}\sup_{t\in(n-1,n]}l(t)\\
\le{}&C\hat u(a_\ast)\int_{-\infty}^{\lceil-\log n_1\rceil}l(t+1)\de t.
\end{align*}
Since $\int_{-\infty}^0l(t)\de t<\infty$ we can choose $n_1\in\N$ such that we have for all $n,m>n_1$,
\[
\bigg|E_{-\log n}\bigg[\sum_{t=0}^{\sigma(n_1)-1}r(\exp(-S_t))\bigg]\bigg|\le\frac{\epsi}{4}.
\]

Considering the first term in \eqref{zentral}, we set
\[
a(n_1,n):=E_{-\log n}H_{\exp(-S_{\sigma(n_1)})}
\]
and claim that there exists $n_0$ such that for all $n,m\ge n_0$ we have $|a(n_1,n)-a(n_1,m)|\le\epsi/2$.
It is
\begin{align*}
|a(n_1,n)-a(n_1,m)|{}={}&\left|E_{-\log n}H_{\exp(-S_{\sigma(n_1)})}-E_{-\log m}H_{\exp(-S_{\sigma(n_1)})}\right|\\
={}&\int H_{\exp(-x)} \left|P_{-\log n}^{S_{\sigma(n_1)}}-P_{-\log m}^{S_{\sigma(n_1)}}\right|(\de x)\\
\le{}&d_\mathrm{TV}\left(P_{-\log n}^{S_{\sigma(n_1)}},P_{-\log m}^{S_{\sigma(n_1)}}\right)\sup_{k\in \{0,\ldots,n_1\}}H_{k}.
\end{align*}
Since $n_1$ is fixed we have $\sup_{k\in\{0,\ldots,n_1\}}|H_{k}|\le C<\infty$ with some constant $C\in\R$.
Lemma \ref{le-roesler} in combination with Lemma \ref{le-tv} yields the claim.

Taking everything into account, we obtain for all $n,m\ge \max\{n_0,n_1\}$ 
\begin{align*}
\left|H_n-H_m\right|{}\le{}&\left|a(n_1,n)-a(n_1,m)\right|+\bigg|E_{-\log n}\bigg[\sum_{t=0}^{\sigma(n_1)-1}r(\exp(-S_t))\bigg]\bigg|\\
&{}+\,\bigg|E_{-\log m}\bigg[\sum_{t=0}^{\sigma(n_1)-1}r(\exp(-S_t))\bigg]\bigg|\\
\le{}&\epsi.
\end{align*}
This shows that $(H_n)_{n\in\N}$ is a Cauchy sequence and thus it converges.
\end{proof}

\begin{proof}{ of Corollary \ref{cor-ipl}}
Parts a), c) and d) of Corollary \ref{cor-ipl} are immediate consequences of Theorem \ref{th} and \citet[Theorem 5.1]{neininger_rueschendorf_99}.
To prove part b), we use that convergence with respect to the $\ell_2$-metric implies convergence of the second moments.
Thus, we obtain as consequence of part a) $\lim_{n\to\infty}E[X_n^2]=E[X^2]$.
Using the distributional fixed point equation characterizing $X$, we have
\begin{align*}
E[X^2]{}={}&E\left[\left(\sum_{k=1}^bV_kX^{(k)}+1+\frac 1\mu\sum_{k=1}^bV_k\log V_k\right)^2\right]\\
={}&\kern-0.1ex\sum_{k=1}^b\kern-0.2ex E[V_k^2]E\kern-0.5ex\left[\kern-0.4ex\big(X^{(k)}\big)^2\right]\kern-0.2ex +\kern-0.2ex E\kern-0.6ex\left[ 1+\frac 2\mu\sum_{k=1}^bV_k\log V_k+\frac 1{\mu^2}\kern-0.4ex\left(\sum_{k=1}^bV_k\log V_k\right)^2\right]
\end{align*}
where we used the independence between $(V_1,\ldots,V_b)$ and $(X^{(1)},\ldots,X^{(b)})$ as well as the fact that $E[X^{(k)}]=0$ for all $k$.
Since $\mu=-bE[V_i\log V_i]$ for all $i=1,\ldots,b$ and $E[X^2]=E[(X^{(k)})^2]=:\sigma^2$ the claim follows.
\end{proof}

\section{The Wiener index}\label{sec-wi}
We now turn to the investigation of the Wiener index.
To handle the Wiener index similarly to the internal path length, we first need a recursion formula for it.
The Wiener index is the sum of the distances between all unordered pairs of balls in the tree.
Let $\Delta_{k,l}$ denote the distance between the balls $k$ and $l$.
Then we have
\[
W_n=\sum_{k<l}\Delta_{k,l}.
\]
Subdividing the sum into the sum for all pairs, where both balls are located in the same subtree, and the sum for all other pairs, we obtain
\[
W_n=\sum_{i=1}^bW_{I_{n,i}}^{(i)}+\sum_{i<j}\sum_{l\in T_{n,j}}\sum_{k\in T_{n,i}}\Delta_{k,l}
\]
where $W^{(i)}_{I_{n,i}}$ denotes the Wiener index of the $i$-th subtree $T_{n,i}$ being of size $I_{n,i}$.
For $k\in T_{n,i}$ and $l\in T_{n,j}$ with $i\not=j$ it is $\Delta_{k,l}=D_k^{(i)}+1+D_l^{(j)}+1$ where $D_k^{(i)}$ is the depth of the ball $k$ with respect to the subtree $T_{n,i}$.
By symmetry of $\Delta_{k,l}$ we can sum up only the first part $D_k^{(i)}+1$ but for all ordered pairs of balls and we obtain
\[
\sum_{i<j}\sum_{l\in T_{n,j}}\sum_{k\in T_{n,i}}\Delta_{k,l}=\sum_{i\not=j}\sum_{l\in T_{n,j}}\sum_{k\in T_{n,i}}(D_k^{(i)}+1).
\]
The summation over $k\in T_{n,i}$ yields
\[
\sum_{i\not=j}\sum_{l\in T_{n,j}}\sum_{k\in T_{n,i}}(D_k^{(i)}+1)=\sum_{i\not=j}\sum_{l\in T_{n,j}}(P_{I_{n,i}}^{(i)}+I_{n,i})
\]
where $P^{(i)}_{I_{n,i}}$ denotes the internal path length of the $i$-th subtree $T_{n,i}$.
Since there are all together $n-I_{n,i}$ balls not lying in $T_{n,i}$, we finally obtain the recursion formula for the Wiener index of the random split tree with $n$ balls: 
\eqn\label{rekursion-wiener-index}
W_n=\sum_{i=1}^b\left[W_{I_{n,i}}^{(i)}+(n-I_{n,i})P_{I_{n,i}}^{(i)}+I_{n,i}(n-I_{n,i})\right].
\nqe

\begin{proof}{ of Theorem \ref{th-wiener}}
Starting from equation \eqref{rekursion-wiener-index} and taking the expectation yields
\eqn\label{exp-wi-1}
E[W_n]=b\sum_{k=0}^{n-s_0}P(I_{n,1}=k)\left(E[W_k]+(n-k)E[P_k]+nk-k^2\right)
\nqe
because all subtrees are identically distributed.
Theorem \ref{th} implies $E[P_k]=\frac 1\mu k\log k+c_pk+o(k)$.
Substituting this in \eqref{exp-wi-1} yields with $E[I_{n,1}]=n/b+o(n)$,
\begin{align}\label{exp-wi-2}
E[W_n]{}={}&b\sum_{k=0}^{n-s_0}P(I_{n,1}=k)E[W_k]+\frac 1\mu b\left(nE[I_{n,1}\log I_{n,1}]-E[I_{n,1}^2\log I_{n,1}]\right)\notag\\
&+\,(c_p+1)n^2-(c_p+1)bE[I_{n,1}^2]+o(n^2).
\end{align}
Substituting the results from Lemma \ref{le-exp-I} in \eqref{exp-wi-2} provides
\begin{align}\label{exp-wi-3}
E[W_n]{}={}&\sum_{k=0}^{n-s_0}bP(I_{n,1}=k)E[W_k]+\frac 1\mu(1-bE[V^2])n^2\log n\notag\\*
&-\,\left(\frac b\mu E[V^2\log V]+bE[V^2]-c_p(1-bE[V^2])\right)n^2+o(n^2).
\end{align}

We set
\[
H_n:=\frac{E[W_n]-\frac 1\mu n^2\log n}{n}.
\]
To prove Theorem \ref{th-wiener} it suffices to show that for each $\epsi>0$ there exists a constant $c\in\R$ and $n_0\in\N$ such that for all $n\ge n_0$
\[
\frac{H_n}{n}\in(c-\epsi,c+\epsi).
\]
So, let $\epsi>0$ be given.
Substituting $H_n$ in \eqref{exp-wi-3} and using Lemma \ref{le-exp-I} yields
\[
H_n=\sum_{k=0}^{n-s_0}\nu_n(\{k\})H_k+r(n)
\]
with
\[
r(n):=-\left(bE[V^2]-c_p(1-bE[V^2])\right)n+o(n).
\]
We set $\tilde d:=-bE[V^2]+c_p(1-bE[V^2])$.
As in the proof of Theorem \ref{th} the conditions of Lemma \ref{le-representation} are fulfilled and we have the representation
\eqn\label{rep-wi}
H_n=E_{-\log n}H_{\exp(-S_{\sigma(n_1)})}+E_{-\log n}\sum_{t=0}^{\sigma(n_1)-1}r(\exp(-S_t)).
\nqe
We start again with the second term and split it in the following way
\begin{align*}
E_{-\log n}\sum_{t=0}^{\sigma(n_1)-1}r(\exp(-S_t)){}={}&E_{-\log n}\sum_{t=0}^{\tau(-\log n+d)}r(\exp(-S_t))\\
&{}+\,E_{-\log n}\sum_{t=\tau(-\log n+d)+1}^{\sigma(n_1)-1}r(\exp(-S_t)).
\end{align*}
For the second summand we obtain by Lemma \ref{le-bruhn-1} with $l(x):=\tilde d\exp(-x)$ and $n_1$ large enough such that $-\log n_1\le a_\ast$
\begin{align*}
0\le \left|E_{-\log n}\sum_{t=\tau(-\log n+d)+1}^{\sigma(n_1)-1}r(\exp(-S_t))\right|{}\le{}&\hat u(a_\ast)\sum_{n=\lfloor -\log n+d \rfloor}^{\lceil-\log n_1\rceil}\sup_{t\in(n-1,n]}|\tilde d|e^{-t}\\
\le{}& C\int_{-\log n+d-3}^{-\log n_1}e^{-x}\,\de x\\
\le{}&C ne^{-d+3}
\end{align*}
with some constant $C$.
We choose $d$ large enough, such that $Ce^{-d+3}<\epsi/3$.
For this $d$ Corollary \ref{co-hauptteil} yields $\hat n_0\in\N$ such that for all $n\ge \hat n_0$
\eqn\label{eq-haupt}
\frac 1nE_{-\log n}\sum_{t=0}^{\tau(-\log n+d)}r(\exp(-S_t))\in\left(c-\frac\epsi 3,c+\frac\epsi 3\right)
\nqe
for some constant $c$.
As in the proof of Theorem \ref{th} the first summand in \eqref{rep-wi} is a Cauchy sequence, i.e.\ there exists $\tilde n_0\in\N$ such that for all $n\ge \tilde n_0$ we have
\[
\left|\frac 1n E_{-\log n}[H_{\exp(S_{\sigma(n_1)})}]\right|<\frac\epsi 3.
\]
Altogether, we have seen that for $n_1\in\N$ with $-\log n_1\le a_\ast$ there exists $n_0\in\N$ such that for all $n\ge n_0$ we have
\begin{align*}
\frac {H_n}{n}{}={}&\frac 1nE_{-\log n}H_{\exp(-S_{\sigma(n_1)})}+\frac 1nE_{-\log n}\sum_{t=0}^{\sigma(n_1)-1}r(\exp(-S_t))\\
\in{}&(c-\epsi,c+\epsi)
\end{align*}
with the constant $c$ in \eqref{eq-haupt}.
Thus, the claim follows.
\end{proof}

\begin{proof}{ of Theorem \ref{th-limit-wiener}}
We define
\begin{align*}
w_n{}:={}&E[W_n]=\frac 1\mu n^2\log n+c_wn^2+o(n^2),\\
p_n{}:={}&E[P_n]=\frac 1\mu n\log n+c_pn+o(n)
\intertext{and}
X_n{}:={}&\left(\frac{W_n-w_n}{n^2},\frac{P_n-p_n}{n}\right)^T.
\end{align*}
For $i\in\{1,\ldots,b\}$ let $X_n^{(i)}$ be an independent copy of $X_n$.
Since the subtrees of the random split tree are independent conditioned upon there sizes, we obtain from \eqref{rekursion-wiener-index} for the standardized vector $X_n$ the following recursion formula
\[
X_n{}\eqd{}\sum_{i=1}^b A_i^{(n)}X_{I_{n,i}}^{(i)}+b^{(n)}
\]
with
\[
A_i^{(n)} :=\begin{bmatrix} \frac 1{n^2}&0\\0&\frac 1n\end{bmatrix}\begin{bmatrix} 1&n-I_{n,i}\\0&1\end{bmatrix}\begin{bmatrix} I_{n,i}^2&0\\0&I_{n,i}\end{bmatrix}=\begin{bmatrix} \frac{I_{n,i}^2}{n^2}&\frac{I_{n,i}(n-I_{n,i})}{n^2}\\[1.2ex]0&\frac{I_{n,i}}{n}\end{bmatrix}
\]
and $b^{(n)}=\left(b_1^{(n)},b_2^{(n)}\right)^T$ where
\begin{align*}
b_1^{(n)}{}={}\frac 1{n^2}&{}\left\{{}\sum_{i=1}^b I_{n,i}\left(n-I_{n,i}\right)-\frac 1\mu n^2\log n-c_wn^2+o(n^2)\right.\\*
&\left.{}+{}\,\sum_{i=1}^bw_{I_{n,i}}+n\sum_{i=1}^b p_{I_{n,i}}-\sum_{i=1}^bI_{n,i}\,p_{I_{n,i}}\right\}
\end{align*}
and
\[
b_2^{(n)}:=1-\frac 1\mu\log n-c_p+o(1)+\frac 1n \sum_{i=1}^bp_{I_{n,i}}+o(1).
\]

Using $\sum_{i=1}^bI_{n,i}=n-s_0$ it follows
\begin{align*}
n\sum_{i=1}^b p_{I_{n,i}}-\frac1\mu n^2\log n{}={}&n\frac 1\mu \sum_{i=1}^bI_{n,i}\log \frac{I_{n,i}}{n}+c_pn(n-s_0)+o(n^2)
\end{align*}
and
\begin{align*}
\sum_{i=1}^bw_{I_{n,i}}-\sum_{i=1}^bI_{n,i}\,p_{I_{n,i}}{}={}&(c_w-c_p)\sum_{i=1}^bI_{n,i}^2+o(n^2).
\end{align*}
This yields with $I_{n,i}=o(n^2)$
\begin{align}\label{eq:2.3}
b_1^{(n)}{}={}&\frac 1\mu \sum_{i=1}^b\frac{I_{n,i}}{n}\log \frac{I_{n,i}}{n}+(1+c_p-c_w)\left(1-\sum_{i=1}^b\frac{I_{n,i}^2}{n^2}\right)+o(1).
\end{align}

By similar arguments we have 
\begin{equation}\label{eq:2.4}
b_2^{(n)}=\frac1\mu\sum_{i=1}^b \frac{I_{n,i}}{n}\log \frac{I_{n,i}}{n}+1+o(1).
\end{equation}

In order  to use the contraction method as in \citet[Theorem 4.1]{neininger_01} it suffices to show that for $n\to\infty$
\begin{equation}\label{a}
\left(\vect{A^{(n)}}{1}{b}, b^{(n)}\right) \stackrel{\ell_2}{\longrightarrow} \left(\vect{A^\ast}{1}{b},b^\ast\right),
\end{equation}
\begin{equation}\label{b}
E\left[\1_{\{I_{n,i}\le l\}\cup\{I_{n,i}=n\}}\left\|(A_i^{(n)})^TA_i^{(n)}\right\|_{\mathrm{op}}\right]\to 0
\end{equation}
for all $l\in\N$ and
\begin{equation}\label{c}
\sum_{i=1}^bE\left\|(A_i^\ast)^TA_i^\ast\right\|_{\mathrm{op}}<1
\end{equation}
where $\|\cdot\|_{\mathrm{op}}$ is the operator norm.

By Lemma \ref{le-conv-I} we know that $I_n/n$ converges in probability to $V:=(\vect{V}{1}{b})$, which is the splitting vector.
By  equations  \eqref{eq:2.3} and \eqref{eq:2.4} we have $b^{(n)}\to b^\ast$ in probability as $n\to\infty$ with
\[
b^\ast=\frac 1\mu \sum_{i=1}^b V_i\log V_i\begin{pmatrix} 1\\1\end{pmatrix}+\begin{pmatrix}( 1+c_p-c_w)\left(1-\sum_{i=1}^bV_i^2\right)\\1\end{pmatrix}.
\]
By the boundedness of the function $x\mapsto x\log x$ on $[0,1]$ and as $I_{n,i}/n\in[0,1]$ there exists a constant $C$ such that
\begin{align*}
\left|b_1^{(n)}\right|{}\le{}& C&\text{and}&&\left|b_2^{(n)}\right|{}\le{}& C.
\end{align*}
Thus, we get the uniform integrability of $(b_1^{(n)})^2$ and $(b_2^{(n)})^2$ and consequently the convergence of $b^{(n)}$ with respect to the $\ell_2$-metric.
Similar arguments yield the convergence of $A_i^{(n)}$ with respect to the $\ell_2$-metric to
\[
A_i^\ast=\begin{bmatrix} V_i^2&V_i(1-V_i)\\0&V_i\end{bmatrix}.
\]
This shows condition \eqref{a}.

Condition \eqref{b} follows from the deterministic boundedness of $\|A_i^{(n)}\|_{\mathrm{op}}$ and from the fact that
\begin{align*}
\lefteqn{\lim_{n\to\infty}P\left(\{I_{n,i}\le l\}\cup\{I_{n,i}=n\}\right)}\qquad\\
{}={}&\lim_{n\to\infty}\int_0^1P(\mathrm{Bin}(\eta_n,x)\le l-s_1)\de P^V(x)\\
\le&{}\lim_{n\to\infty}P\left(V\le\left((l-s_1)/\eta_n\right)^{\frac13}\right)\\
&{}+\lim_{n\to\infty}\int_{\left(\frac{l-s_1}{\eta_n}\right)^{\frac 13}}^1\exp\left(-\frac 14\eta_n^{\frac 13}(l-s_1)^{\frac 23}\left(1-\left(\frac{l-s_1}{\eta_n}\right)^{\frac 23}\right)^2\right)\de P^V(x)\\
={}&0
\end{align*}
where we used Bernstein's inequality.

It remains to show \eqref{c}.
Solving the characteristic equation for the matrix $(A_i^\ast)^TA_i^\ast$ we obtain that its eigenvalue $\lambda(V_i)$ being larger in absolute value is given by
\[
\lambda(V_i)=V_i^2\left(1-V_i+V_i^2+(1-V_i)\sqrt{V_i^2+1}\right).
\]
Elementary calculations show $x>x^2(1-x+x^2+(1-x)\sqrt{x^2+1})$ for all $x\in(0,1)$.
Thus, we have $E[\lambda(V_i)]<E[V_i]=1/b$ because it is $P(V_i\in\{0,1\})=0$.
This finally implies
\[
E\left[\sum_{i=1}^b\left\|(A_i^\ast)^TA_i^\ast\right\|_{\mathrm{op}}\right]= E\left[\sum_{i=1}^b\lambda(V_i)\right]<1.
\]
The claim for the asymptotic behavior of the variance of $W_n$ follows directly from the first part, since convergence with respect to the $\ell_2$-metric implies convergence of the second moments.
\end{proof}

\bibliographystyle{plainnat}

\appendix
\section{Proof of Lemma \ref{le-roesler}}\label{app-a}
We give the essential parts of \citet{roesler_91} which prove Lemma \ref{le-roesler}.
\begin{proof}{ of Lemma \ref{le-roesler}}
Let $a\in\R_-$.
We use the notation 
\[
\Delta(a):=\lim_{x_0\to-\infty}\sup_{x,y\le x_0}d_\mathrm{TV}\left(P_x^{S_{\tau(a)}},P_y^{S_{\tau(a)}}\right).
\]
Since the function 
\[
x_0\quad\mapsto\quad\sup_{x,y\le x_0}d_\mathrm{TV}\left(P_x^{S_{\tau(a)}},P_y^{S_{\tau(a)}}\right)
\]
is increasing and non-negative, the limit for $x_0\to-\infty$ exists.
We will show that $\Delta(a)\le (1-\tilde\epsi)\Delta(a)+\delta$ for some $\tilde\epsi>0$ and all $\delta>0$.
Then the claim follows.

Let $\delta>0$ be an arbitrary number.
Since the process $S$ fulfills the integrability condition and $S_{\tau(y)}-y\le S_{\tau(y)}-S_{\tau(y)-1}$, there exists $x_1\in\R_-$ such that for all $x<y<x_1$
\[
E_x[S_{\tau(y)}-y]\;\le\;\int z\;\de\bar F_y(z)\;\le\; C<\infty
\]
for some constant $C$.
Thus, there exists $K_1\ge K$ such that for all $y<x_1$
\eqn\label{term_3}
P_x(S_{\tau(y)}-y> K_1)\le\frac{E_x[S_{\tau(y)}-y]}{K_1}\le \frac{\delta}{4}.
\nqe

Furthermore, we have for this $K_1$
\eqn\label{term_2}
\sup_{y+K\le z\le y+K_1}d_\mathrm{TV}\left(P_z^{S_{\tau(a)}},P_y^{S_{\tau(a)}}\right)\le\sup_{u,v\le y+K_1}d_\mathrm{TV}\left(P_u^{S_{\tau(a)}},P_v^{S_{\tau(a)}}\right).
\nqe

The distribution of the Markov chain $S$ on the state space $\cE$ is given by the kernel
\[
\kappa(x,A):=P(S_{t+1}\in A\mid S_t=x)\qquad\text{for all $t\in\N_0$ and $A\subset\cE$.}
\]
Let $S^{(a)}$ be the process $S$ stopped at the moment when it exceeds $a\in\cE$.
The kernel $\kappa_a$ corresponding to the process $S^{(a)}$ is then given by $\kappa_a(x,A)=\kappa(x,A)$ for $x\le a$ and $\kappa_a(x,A):=\1_A(x)$ for $x>a$ and for all $A\subset\cE$.

Let $D:=\{(x,x)\mid x\in\cE\}$ denote the diagonal in $\cE^2$.
We define a kernel $\varrho$ on $\cE^2$ by the so called Wasserstein coupling \citep[see e.g.][]{griffeath_74}, i.e. for $(x,y),(u,v)\in \cE^2$ it is
\[
\varrho((x,y),(u,v)):=
\begin{cases}
\min\{\kappa_a(x,u),\kappa_a(y,v)\},&\text{if $u=v$}\\
\frac {(\kappa_a(x,u)-\kappa_a(y,u))^+(\kappa_a(y,v)-\kappa_a(x,v))^+}{1-\alpha(x,y)},&\text{if $u\not= v$}
\end{cases}
\]
where $\alpha(x,y):=\sum_{z\in\cE}\min\left\{\kappa_a(x,z),\kappa_a(y,z)\right\}$ and $r^+=\max\{r,0\}$ denotes the positive part of a real number $r$.
Then the following properties hold:
\begin{enumerate}
\item $\varrho((x,y),A\times\cE)=\kappa_a(x,A)$ and $\varrho((x,y),\cE\times A)=\kappa_a(y,A)$ for all $x,y\in\cE$ and $A\subset\cE$
\item $\varrho((x,x),D)=1$ for all $x\in\cE$ and
\item $\varrho((x,y),D^c)\le 1-\epsi$ for all $x,y\in\cE$ with $|x-y|\le K$ and $x,y<x_0$.
\end{enumerate}
The property c) follows from the assumption \eqref{cond-tv} and the fact that
\begin{align*}
d_\mathrm{TV}\left(P_x^{S_1},P_y^{S_1}\right){}={}&\sum_{z\in E}\left|\kappa_a(x,z)-\kappa_a(y,z)\right|\\
={}&2\left(1-\sum_{z\in E}\min\{\kappa_a(x,z),\kappa_a(y,z)\}\right).
\end{align*}

For $(x,y)\in\cE^2$ let $Z^{(x,y)}=(U^{(x,y)},V^{(x,y)})$ be the Markov chain generated by the kernel $\varrho$ which starts in $(x,y)$. We define the stopping time 
\[
\theta(a):=\inf\{t\mid Z^{(x,y)}_t\in(a,\infty)\times(a,\infty)\}.
\]

Using this coupling we obtain for any $K_2>0$ and $z,y<a$
\begin{align}\label{term_1}
\lefteqn{d_\mathrm{TV}\left(P_z^{S_{\tau(a)}},P_y^{S_{\tau(a)}}\right)}\qquad\notag\\
={}&\sum_{w\in\cE}\left|P_z(S_{\tau(a)}=w)-P_y(S_{\tau(a)}=w)\right|\notag\\
={}&\sum_{w\in\cE}\bigg|P\left(U^{(z,y)}_{\theta(a)}=w\right)-P\left(V^{(z,y)}_{\theta(a)}=w\right)\bigg|\notag\\
={}&\sum_{(u,v)\in\cE^2}\sum_{w\in\cE}P\left(Z^{(z,y)}_1=(u,v)\right)\notag\\
&\hspace*{0.5cm}\times\bigg|\underbrace{P\left(U^{(z,y)}_{\theta(a)}=w\mid Z^{(z,y)}_1=(u,v)\right)}_{=P_u(S_{\tau(a)}=w)}-\underbrace{P\left(V^{(z,y)}_{\theta(a)}=w\mid Z^{(z,y)}_1=(u,v)\right)}_{=P_v(S_{\tau(a)}=w)}\bigg|\notag\\[1.2ex]
\le{}&\kern-1ex\sup_{u,v\le y+K_2}\kern-1ex d_\mathrm{TV}\left(P_u^{S_{\tau(a)}},P_v^{S_{\tau(a)}}\right)\varrho((z,y),D^c)+2P\left(Z^{(z,y)}_1\notin(-\infty,y+K_2]^2\right).
\end{align}
In the last step we used that $P_u(S_{\tau(a)}=w)-P_v(S_{\tau(a)}=w)=0$ for $u=v$.
As seen in equation \eqref{term_3} and using property a) of the coupling, there exists by the integrability condition $K_2>K$ such that for all $y<x_1-K$ and $y<z<y+K$
\begin{align}\label{term_1_c}
P\left(Z^{(z,y)}_1\notin(-\infty,y+K_2]^2\right){}\le{}& \kappa_a\left(z,(-\infty,y+K_2]^c\right)+\kappa_a\left(y,(-\infty,y+K_2]^c\right)\notag\\
{}\le{}&\frac{\delta}{4}.
\end{align}

After these preliminaries, we now turn to $\Delta(a)$.
It is for $x<y<a-K$
\begin{align*}
d_\mathrm{TV}\left(P_x^{S_{\tau(a)}},P_y^{S_{\tau(a)}}\right)
={}&\int d_\mathrm{TV}\left(P_z^{S_{\tau(a)}},P_y^{S_{\tau(a)}}\right)\de P_x^{S_{\tau(y)}}(z)\\
={}&\int_{[y,y+K]} d_\mathrm{TV}\left(P_z^{S_{\tau(a)}},P_y^{S_{\tau(a)}}\right)\de P_x^{S_{\tau(y)}}(z)\\
&{}+\,\int_{(y+K,y+K_1]} d_\mathrm{TV}\left(P_z^{S_{\tau(a)}},P_y^{S_{\tau(a)}}\right)\de P_x^{S_{\tau(y)}}(z)\\
&{}+\,\int_{(y+K_1,\infty)} d_\mathrm{TV}\left(P_z^{S_{\tau(a)}},P_y^{S_{\tau(a)}}\right)\de P_x^{S_{\tau(y)}}(z)\\
\le{}&P_x(S_{\tau(y)}-y\le K)\sup_{y\le z\le y+K}d_\mathrm{TV}\left(P_z^{S_{\tau(a)}},P_y^{S_{\tau(a)}}\right)\\
&{}+\,P_x(S_{\tau(y)}-y> K)\sup_{y+K\le z\le y+K_1}d_\mathrm{TV}\left(P_z^{S_{\tau(a)}},P_y^{S_{\tau(a)}}\right)\\
&{}+\,2P_x(S_{\tau(y)}-y\ge K_1).
\end{align*}
With the results in \eqref{term_3}, \eqref{term_2}, \eqref{term_1} and \eqref{term_1_c} as well as property c) of the kernel $\varrho$ this finally yields
\begin{align*}
\Delta(a){}\le{}&\lim_{x_0\to-\infty}\sup_{x<y\le x_0}\left[P_x(S_{\tau(y)}-y\le K)\kern-.4ex\sup_{u,v\le y+K_2}d_\mathrm{TV}\left(P_u^{S_{\tau(a)}},P_v^{S_{\tau(a)}}\right)\left(1-\epsi\right)\right.\\
&+\,P_x(S_{\tau(y)}-y> K)\sup_{u,v\le y+K_1}d_\mathrm{TV}\left(P_z^{S_{\tau(a)}},P_y^{S_{\tau(a)}}\right)\\
&\left.+\,2P\left(Z^{(z,y)}_1\notin(-\infty,y+K_2]^2\right)\right]+\frac{\delta}{2}\\
\le{}&\Delta(a)\lim_{x_0\to-\infty}\sup_{x<y\le x_0}\left(1-\epsi P_x(S_{\tau(y)}-y\le K)\right)+\delta\\
\le{}&\left(1-\tilde\epsi\right)\Delta(a)+\delta
\end{align*}
where $\tilde\epsi=\epsi\lim_{x_0\to-\infty}\inf_{x<y\le x_0} P_x(S_{\tau(y)}-y\le K)>0$.
\end{proof}

\section{Proofs from \citet{bruhn_96}}\label{app-b}
\begin{proof}{ of Lemma \ref{le-representation}}
For $n\le n_1$ the claim follows immediately since $\sigma(n_1)=0$.
For $n>n_1$ equation \eqref{zentral} follows by induction on $n$.
It is with $H_{1/e}:=H_0$
\begin{align*}
H_{n+1}{}={}&\sum_{k=0}^n\nu_{n+1}(\{k\})H_k+r(n+1)\\
={}&\sum_{k=1}^nP_{-\log (n+1)}(S_1=-\log k)\,E_{-\log k}\kern-1ex\left[H_{\exp(-S_{\sigma(n_1)})}+\kern-2ex\sum_{t=0}^{\sigma(n_1)-1}\kern-1ex r(\exp(-S_t))\kern-.5ex\right]\\
&+E_{-\log(n+1)}[r(\exp(-S_0))]+P_{-\log(n+1)}(S_1=1)E_1[H_{\exp(-S_{\sigma(n_1)})}]\\
={}&E_{-\log(n+1)}H_{\exp(-S_{\sigma(n_1)})}+E_{-\log(n+1)}\left[\sum_{t=0}^{\sigma(n_1)-1}r(\exp(-S_t))\right]
\end{align*}
where we use the Kolmogorov-Chapman equation for Markov chains in the last step.
\end{proof}

\begin{proof}{ of Lemma \ref{le-bruhn-1}}
We use the notation from Section \ref{sec-bruhn} and define for $x\in\R_-$ the function $u_x$ by
\[
u_x(a):=E_x[|\{t:S_t\in(a,a+1]\}|].
\]
By the monotone convergence theorem we have $\lim_{a\to-\infty}E[\underline Y_t^{(a)}]=E[\tilde Y_t]>0$.
Thus, there exists $a_\ast\in\R$ such that for all $a<a_\ast$ it is $E[\underline Y_t^{(a)}]>0$.
For $x,n,a<a_\ast$ and $k\in\N$ it holds
\begin{align*}
P_x(|\{t:S_t\in(n-1,n]\}|\ge k){}={}&\int_{(n-1,n]}P_y(S_{k-1}\le n)\;\de P_x^{S_{\tau(n-1)}}(y)\\
\le{}&\int_{(n-1,n]}P_y(\underline S^{(a)}_{k-1}\le n)\;\de P_x^{S_{\tau(n-1)}}(y)\\
\le{}&\int_{(n-1,n]}P_0(\underline S^{(a)}_{k-1}\le 1)\;\de P_x^{S_{\tau(n-1)}}(y)\\
\le{}&P_0(\underline S^{(a)}_{k-1}\le 1)\\
={}&P_0(|\{t:\underline S^{(a)}_t\in[0,1]\}|\ge k).
\end{align*}
Thus, we have
\begin{align*}
u_x(n-1){}\le{}&\sum_{k=1}^\infty P_0(|\{t:\underline S^{(a)}_t\in[0,1]\}|\ge k)\\
={}&E_0[|\{ t:\underline S^{(a)}_t\in[0,1]\}|]\\
=:{}&\hat u(a).
\end{align*}
Since it is $E[\underline{Y}^{(a)}_t]>0$ the elementary renewal theorem \citep[see e.g.][Section II.4]{gut_88} provides $\hat u(a)<\infty$.
Furthermore, the function $a\mapsto\hat u(a)$ is decreasing as $a\to-\infty$, i.e.\ $\hat u(a)\le\hat u(a_\ast)$ for all $a<a_\ast$.

So we finally obtain for a function $l:\R\to\R_+$, $y,z\in\R$ and $x\in \cE$ with $x<y<z<a_\ast$
\begin{align*}
E_x\left[\sum_{t=\tau(y)}^{\tau(z)-1}l(S_t)\right]{}\le{}&\sum_{n=\lceil y\rceil}^{\lceil z\rceil}u_x(n-1)\sup_{t\in(n-1,n]}l(t)\\
\le{}&\hat u(a_\ast)\sum_{n=\lceil y\rceil}^{\lceil z\rceil}\sup_{t\in(n-1,n]}l(t).
\end{align*}
\end{proof}
\end{document}